\documentclass[11pt]{amsart}
\usepackage{amssymb,amsmath,amsrefs}
\usepackage{xcolor}

\newtheorem{thm}{Theorem}
\newtheorem{lemma}[thm]{Lemma}

\newtheorem{corollary}[thm]{Corollary}

\theoremstyle{definition}
\newtheorem{remark}[thm]{Remark}

\def\o {\mbox{Orb}}

\def\t {\tau}
\def\b {\beta}
\def\d {\delta}
\def\p {\mbox{Per}}
\def\a {\alpha}

\newcommand{\Z}{\mathbb{Z}}

\newcommand{\Q}{\mathbb{Q}}
\newcommand{\QQ}{\overline{\Q}}
\newcommand{\R}{\mathbb{R}}
\newcommand{\C}{\mathbb{C}}

\newcommand{\e}{\epsilon}

\newcommand{\g}{\gamma}

\newcommand{\E}{\mathcal{E}}
\newcommand{\T}{\mathcal{T}}

\newcommand{\M}{\mathcal M}

\begin{document}

\title{Algebraic periodic points of transcendental entire functions}

\author[D. Krumm]{David Krumm}
\address{Escuela de Matemática\\ Universidad de Costa Rica\\
San Jos\'e, Costa Rica}
\email{david.krumm@ucr.ac.cr}
\urladdr{http://maths.dk}

\author[D. Marques]{Diego Marques}
\address{Departamento de Matem\'atica\\
Universidade de Bras\'ilia\\
Bras\'ilia, DF\\
Brazil}
\email{diego@mat.unb.br}

\author[C. G. Moreira]{Carlos Gustavo Moreira}
\address{Instituto de Matem\'atica Pura e Aplicada\\
Rio de Janeiro, RJ\\
Brazil}
\email{gugu@impa.br}

\author[P. Trojovsk\'y]{Pavel Trojovsk\'y}
\address{Faculty of Science, University of Hradec Kr\'alov\'e\\
Czech Republic}
\email{pavel.trojovsky@uhk.cz}


\begin{abstract}
We prove the existence of transcendental entire functions $f$ having a property studied by Mahler, namely that $f(\QQ)\subseteq \QQ$ and $f^{-1}(\QQ)\subseteq \QQ$, and in addition having a prescribed number of $k$-periodic algebraic orbits, for all $k\geq 1$. Under a suitable topology, such functions are shown to be dense in the set of all entire transcendental functions.
\end{abstract}

\maketitle

\section{Introduction}\label{intro_section}

Let $D$ be a subset of $\C$. Recall that a function $f:D\to\C$ is called \emph{transcendental} if the only two-variable polynomial $P\in\C[x,y]$ such that $P(z,f(z))=0$ for all $z\in D$ is the zero polynomial. For the purposes of this article, let us define a \emph{Mahler function} to be a transcendental entire function $f:\C\to\C$ with the property that
\begin{equation}\label{Qbar_invariance}
f(\overline\Q)\subseteq\overline\Q\quad\text{and}\quad f^{-1}(\overline\Q)\subseteq\overline\Q,
\end{equation}
where $\overline\Q$ denotes the field of the algebraic numbers, i.e., the algebraic closure of $\Q$ in $\C$. If $f$ is a Mahler function and $K$ is a subset of $\C$ such that $K\cap\overline\Q$ is dense in $\C$ (for instance, $K=\Q(i)$), we say that $f$ is \emph{defined over K} if every coefficient of the Taylor series
\[
f(z)=\sum_{n=0}^\infty a_nz^n
\]
belongs to $K$. This terminology is motivated by a question posed by Mahler \cite{mahler}*{p. 53} in 1976: \textit{Does there exist a transcendental entire  function $f$ for which \eqref{Qbar_invariance} holds, and such that every coefficient of its Taylor series is rational?} Mahler mentions that the answer seems to be unknown even if the coefficients are allowed to be arbitrary \emph{complex} numbers. 

The question of existence of Mahler functions was settled in recent work of the second and third authors \cite{marques-moreira1}, who proved that there exist uncountably many Mahler functions defined over $\Q$, thus answering Mahler's question. In this article we study some of the dynamical properties of Mahler functions.

By virtue of the property \eqref{Qbar_invariance}, every Mahler function $f$ can be regarded as a discrete dynamical system on $\overline\Q$. One can then ask how many algebraic periodic orbits $f$ has of every period. Our main result is Theorem \ref{main_thm} below, which states, loosely speaking, that given any transcendental entire function $g$, one can obtain -- via a small perturbation of $g$ -- a Mahler function having prescribed dynamical behavior on $\overline\Q$. 

To state our results precisely, we introduce some notation and terminology. For every dynamical system $f:\overline\Q\to\overline\Q$, we denote by $\p(k, f)$ the set of all $k$-periodic points of $f$, i.e., the algebraic numbers $\alpha$ for which $k$ is the smallest positive integer such that $f^k(\alpha)=\alpha$ (where $f^1=f$ and $f^k=f\circ f^{k-1}$ for $k\geq 2$). The \emph{orbit} of an algebraic number $\alpha$ is the set $\{\a, f(\a),f^2(\a),\ldots\}$. The orbit of a $k$-periodic point is a \textit{$k$-cycle}. Finally, we denote by $\o(k, f)$ the set of all $k$-cycles of $f$. Note that $\# \o(k, f) = \# \p(k, f)/k$.

Let $\E$ denote the set of all entire functions. We place a topology on $\E$ with a sub-base consisting of sets $V_{g,\vartheta}$ for every function $g\in\E$ and every sequence $\vartheta$ of positive real numbers: if $g(z)=\sum_{n=0}^\infty a_nz^n$ and $\vartheta=(\theta_k)_{k\geq 0}$, the open set $V_{g,\vartheta}$ consists of all functions $\sum_{n=0}^{\infty}b_nz^n \in \E$ satisfying $|a_n-b_n|<\theta_n$ for all $n$. Note that, in this topology, the set $\T$ of all transcendental entire functions is open and dense in $\E$, since every non-polynomial entire function is transcendental.

\begin{thm}\label{main_thm}
Let $K$ be a subset of $\C$ such that $K\cap\overline\Q$ is dense in $\C$ (in the Euclidean topology), and let $\sigma=(s_k)_{k\ge 0}$ be a sequence in $\Z_{\geq 0}\cup \{\infty\}$. Let $\M(\sigma,K)$ denote the set of all Mahler functions $f$ that are defined over $K$ and satisfy $\#\o(k,f)=s_k$ for all $k\ge 0$. Then $\M(\sigma,K)$ is dense in $\T$ (and thus in $\E$).
\end{thm}

Taking $K=\C$ in the theorem, and choosing an arbitrary sequence $\sigma$, we obtain the following new result illustrating the ubiquity of Mahler functions. 

\begin{corollary} The set of all Mahler functions is dense in $\E$. 
\end{corollary}

The proof of our main theorem relies primarily on a theorem of Bergweiler \cites{bergweiler1,bergweiler2} concerning periodic orbits of transcendental entire functions, and builds on techniques developed by the second and third authors in \cites{marques-moreira1,marques-moreira2}.

\subsection*{Acknowledgements} We would like to thank Moubariz Garaev, Harald Helfgott and Lola Thompson for organizing and inviting the three first authors to the meeting Number Theory in the Americas/Teoría de N\'umeros en Am\'erica held in Casa Matem\'atica Oaxaca, in 2019, where this work started.

\section{Proof of Theorem \ref{main_thm}}\label{proof}

\subsection{Auxiliary results}

Throughout this section we use the familiar notation $[a, b] = \{a, a + 1,\ldots, b\}$ for integers $a < b$.

We begin our discussion of the proof of Theorem \ref{main_thm} by stating several preliminary results that will be used in the proof. The first result concerns holomorphic dynamics, and is due to Bergweiler \cite{bergweiler1}. 

Recall that a point $\a\in \p(k, f)$ is called {\it repelling} if $|(f^k)'(\a)|>1$.

\begin{lemma}[Bergweiler]\label{l1}
Every transcendental entire function has infinitely many repelling $k$-periodic points, for all $k\geq 2$.
\end{lemma}

In order to construct prescribed $1$-periodic points (which are not covered by the previous result -- for example, $e^z+z\in \T$ does not have fixed points), we shall need the classical Picard theorems.

\begin{lemma}[Great Picard Theorem]\label{l3}
Suppose that a holomorphic function $f$ (whose domain contains a punctured disk centered at $w$) has an essential singularity at $\omega$. Then, on any punctured neighborhood of $\omega$, $f$ assumes all possible complex values, with at most a single exception, infinitely often.
\end{lemma}

\begin{remark}\label{R1}
Note that every function $g\in \T$ has an essential singularity at infinity. Moreover, if $P(z)$ is any nonzero polynomial, then $g(z)/P(z)$ has an essential singularity at infinity (equivalently, $g(1/z)/P(1/z)$ has an essential singularity at $0$). The previous result implies that $g(z)/P(z)$ assumes all possible complex values, with at most a single exception, infinitely often.
\end{remark}

\begin{lemma}\label{l4}
Let $g(z)$ be an entire function and set $f(z)=g(z)+\e P(z)$, where $\e>0$ is a real number and $P(z)\in \C[z]$ is a nonzero polynomial. Then, for all $k\geq 1$, we have that $f^k(z)=g^k(z)+\e \phi_k(\e, z)$, where $\phi_{k}(\e,z)$ is a nonzero analytic function in both variables $\e\in \C\backslash \{0\}$ and $z\in \C$.
\end{lemma}
\begin{proof}
The proof is by induction on $k$. The base step $k=1$ follows immediately by choosing $\phi_{1}(\e,z)=P(z)$. Suppose that the result is valid for $k$; then 
\begin{eqnarray*}
f^{k+1}(z) & = & f(f^k(z))=f(g^k(z)+\e \phi_{k}(\e,z))\\
 & = & g(g^k(z)+\e \phi_{k}(\e, z))+\e P(g^k(z)+\e \phi_{k}(\e,z)).
\end{eqnarray*}
Now use the identity 
$$g(y+h)-g(y)=h\int_0^1 g'(y+th)dt,$$
which follows from the fact that the derivative of the function $u(t)=g(y+th)$ is $h\cdot g'(y+th)$, to write 
$$g(g^k(z)+\e \phi_{k}(\e, z))=g^{k+1}(z)+\e \phi_{k}(\e, z)\int_0^1 g'(g^k(z)+t\e \phi_{k}(\e, z))dt,$$
and the result follows by setting 
$$\phi_{k+1}(\e, z)=\phi_{k}(\e, z)\int_0^1 g'(g^k(z)+t\e \phi_{k}(\e, z))dt+P(g^k(z)+\e \phi_{k}(\e,z)).$$
\end{proof}

\begin{lemma}\label{l5}
Let $g(z)$ be a transcendental entire function, $P(z)\in \C[z]$ and $(k,M)\in \Z_{\geq 1}^2$. Then, for every $\delta>0$, there are $\delta_1, \delta_2$ with $0<\delta_1<\delta_2<\delta$ and $R\in \R_{>M}$ such that, for all $\e\in (\delta_1,\delta_2)$, the function $f(z):=g(z)+\e P(z)$ has at least $M$ fixed points $w_1,w_2,\ldots,w_M$ with $f'(w_j)\not\in \{0,1\}$ for $j\le M$ and $M$ expansive $t$-cycles entirely contained in $B(0,R)$ for all $t\in [1,k]$.
\end{lemma}
\begin{proof}
We claim, first of all, that $g'(z)P(z)-g(z)P'(z)$ is a transcendental entire function. Indeed, this function is equal to $P(z)^2\cdot(g(z)/P(z))'$, so, if it is a polynomial, then $u(z):=(g(z)/P(z))'$ is a rational function; however, in that case, since $g(z)/P(z)$ is holomorphic in a neighborhood of the set $\{z\in\C: |z|\ge T\}$ for some $T>0$, $|g(z)/P(z)|$ is bounded in $\{|z|=T\}$, so there is $m\geq 1$ such that $|u(z)/z^m|$ is bounded for $|z|\ge T$ and, for $|z|>T$, \[g(z)/P(z)=g(Tz/|z|)/P(Tz/|z|)+\int_{T}^{|z|}u(sz/|z|)\cdot\frac{z}{|z|}ds.\] It follows that $|\frac{g(z)}{P(z)z^{m+1}}|$ is bounded for $|z|\ge T$, but since $P(z)$ is a polynomial, this implies that $g(z)$ is a polynomial, which is a contradiction.

Next, by Remark \ref{R1}, the meromorphic function $h(z):=(z-g(z))/P(z)$ has an essential singularity at infinity. We deduce from Lemma \ref{l3} that there exists $\hat\delta\in (0,\delta)$ such that $h^{-1}(\e)$ is an infinite set for all $\e\in (0,\hat\delta)$. Observe that any element of $h^{-1}(\e)$ is a fixed point of $f(z)=g(z)+\e P(z)$. Now, given a constant $c$, the set of $\e\in\C$ such that, for some $z$ with $P(z)\ne 0$ we have $g(z)+\e P(z)=z$ and $g'(z)+\e P'(z)=c$, is countable. Indeed, these identities imply $g(z)P'(z)+\e P(z)P'(z)=zP'(z)$ and $g'(z)P(z)+\e P(z)P'(z)=cP(z)$, and therefore $$g'(z)P(z)-g(z)P'(z)=cP(z)-zP'(z).$$
The function on the left-hand side of the above identity is transcendental (by the claim proved earlier), while the function on the right-hand side is a polynomial. Hence, the set of solutions of this identity has no accumulation points, and is thus countable, and since $\e=(z-g(z))/P(z)$, the set of corresponding values of $\e$ is also countable. Given that the set of zeros of $P(z)$ is finite, this implies that there is $\tilde\delta\in (0,\hat\delta)$ such that, defining $\tilde f(z)=g(z)+\tilde\delta P(z)$, we have that $\tilde f(z)$ has infinitely many fixed points $w_j$, $j\ge 1$, with $\tilde f'(w_j)\not\in \{0,1\}$ for all $j$.

Bergweiler's theorem \ref{l1} implies that the transcendental entire function $\tilde f(z)=g(z)+\tilde\delta P(z)$ has infinitely many expansive $t$-cycles for every $t$ with $t\in [2, k]$. We can choose $R>0$ such that $\tilde f$ has at least $M$ fixed points as before, which are not zeros of $P(z)$, and $M$ expansive $t$-cycles contained in $B(0,R)$ for every $t\in [2, k]$. Since those fixed points and $t$-cycles persist for small perturbations of $\tilde f$, there is $\eta>0$ such that $(\tilde\delta-\eta,\tilde\delta+\eta)\subset (0,\delta)$ and, for all $\e\in (\delta_1,\delta_2)=(\tilde\delta-\eta,\tilde\delta+\eta)$, the function $f(z):=g(z)+\e P(z)$ has at least $M$ fixed points $w_1,w_2,\ldots,w_M$ with $f'(w_j)\not\in \{0,1\}$ for $j\le M$ and $M$ expansive $t$-cycles entirely contained in $B(0,R)$ for all $t\in [1,k]$.
\end{proof}

\begin{lemma}\label{l6}
Let $R\in (0,+\infty)$, $g(z)$ be an entire function and $P(z)\in \C[z]$ a nonzero polynomial. Suppose that $\a$ is a complex number such that $g(z)\neq \a$ for all $z\in \partial B(0,R)$.Then there exists a positive real number $\delta$ such that, for all $\e\in (0,\delta)$, the function $f(z):=g(z)+\e P(z)$ satisfies
\[
\#(f^{-1}(\a)\cap B(0,R))=\#(g^{-1}(\a)\cap B(0,R)).
\]
\end{lemma}
\begin{proof}
The hypothesis implies that $\min_{z\in \partial B(0,R)}|g(z)-\a|>0$. Since also $\max_{z\in \partial B}|P(z)|>0$, we may define $\delta>0$ by
\[
\delta:=\dfrac{\min_{z\in \partial B}|g(z)-\a|}{\max_{z\in \partial B}|P(z)|}.
\]
Clearly, for any $\e\in (0,\delta)$, we have $|\e P(z)|<|g(z)-\a|$ for all $z\in \partial B(0,R)$. By Rouch\' e's theorem, the functions $g(z)-\a+\e P(z)$ and $g(z)-\a$ have the same number of zeros in $B(0,R)$. Setting $f(z)=g(z)+\e P(z)$, we conclude that $\#(f^{-1}(\a)\cap B(0,R))=\#(g^{-1}(\a)\cap B(0,R))$, as desired.
\end{proof}

Having proved all the necessary auxiliary results, we now proceed to the proof of Theorem \ref{main_thm}.

\subsection{The proof}

The proof of Theorem \ref{main_thm} adapts the perturbation techniques from \cite{marques-moreira2}, with the additional ingredient of constructing the required number of algebraic periodic orbits (for which we use the above mentioned results by Bergweiler), and avoiding the creation of undesirable additional algebraic periodic orbits. These constructions are done by an inductive process, which depends on an enumeration of the algebraic complex numbers and on a suitable growing sequence of balls.

Let $g(z)=\sum_{k\geq 0}b_kz^k\in\T$ and let $(\theta_k)_{k\ge 0}$ be a sequence of positive real numbers. We shall prove the existence of an entire function $\phi$ such that $f:=g+\phi$ is a Mahler function defined over $K$, and that moreover, $\#\o(k,f)=s_k$ and $|a_k-b_k|<\theta_k$ for all $k\geq 0$, where $f(z)=\sum_{k\geq 0}a_kz^k$. We may assume $\theta_k<1/k!, \forall k\ge 0$, so $f(z)$ will automatically be an entire function.
 
Several functions appearing in the process of constructing $f$ will not be Mahler functions, so we make the following general definition. For an arbitrary function $f:\C\to \C$, we say that a $k$-periodic orbit is {\it algebraic} if every term in the cycle is an algebraic number.

Let $\{\a_1,\a_2,\a_3,\ldots\}$ be an enumeration of $\QQ$ with $\a_1=0$. In what follows, we denote by $L(P)$ the \textit{length} of a polynomial $P$, i.e., the sum of the absolute values of the coefficients of $P$.

We construct our desired function inductively. Choose $\e_{0}\in B(0,\theta_0)$and $\e_{1}\in B(0,\theta_1)$ such that $b_0+\e_0\in (K\backslash \{0\})\cap \overline \Q$ and $b_1+\e_1\neq 0$, and set $f_1(z):=g(z)+\e_0+\e_1 z$. Choose also a sequence $(r_m)_{m\ge 1}$ of real numbers with $r_m>\max\{m-1,r_{m-1}\}, \forall m\ge 1$ and $r_1>0$ small such that, for any $z\in \overline{B(0,r_1)}$, $f(z)\neq 0$ and $f(z)\neq z$. We wish to recursively construct a sequence of entire functions $f_2(z), f_3(z), \ldots$ of the form
\[
f_{m+1}(z)=f_m(z)+z^{m+2}h_m(z)P_{m}(z),
\]
with the following properties being satisfied:
\begin{itemize}
\item[(i)] $h_m, P_m\in \C[z]$ and $f_m(z)=g(z)+\sum_{i=0}^{t_m}a_iz^i$ with $t_m\ge m$;
\item[(ii)] Letting $X_m=\{\a_2,\a_3\ldots, \a_{m}\}$, \[\tilde X_m=\{\tau\in f^{-1}_m(\{\a_1,\a_2,\ldots, \a_{m}\})\cap B(0,r_{m})\,|\,\tau\ne 0, \tau\in \QQ, f_m'(\tau)\neq 0\},\] and
\[
Y_m=\bigcup_{k=1}^{m}Y_m^{(k)}
\] 
where, for $k\in [1, m]$, $Y_m^{(k)}$ is a union of $\min\{m,s_k\}$ repelling algebraic periodic orbits of period $k$  of $f_m$ contained in $B(0,r_{m})$, we have $P_m(z)=\prod_{\tau \in X_m \cup \tilde X_m\cup Y_m}(x-\tau)^2$. Moreover, $P_m(z)\mid P_{m+1}(z)$ for all $ m\ge 1$;
\item[(iii)]$f_m(\tau)\in \QQ, f_m'(\tau)\ne 0\ \forall \tau\in X_m$ and, for each integer $j\in [1, m]$, there is $\tau\in \tilde X_m$ such that $f_m(\tau)=\a_j$;
\item[(iv)] $0<L(h_m P_m)<\nu_m:=\frac{1}{m^{m+2+\deg (h_m P_m)}}$; 
\item[(v)] $a_k+b_k\in K$ and $|a_k|<\theta_k$, for $k\in [1, m]$;
\item[(vi)] $f^{-1}_{m+1}(\{\a_1,\a_2,\ldots, \a_{m}\})\cap B(0,r_m)=f^{-1}_m(\{\a_1,\a_2,\ldots, \a_{m}\})\cap B(0,r_m)=f^{-1}_m(\{\a_1,\a_2,\ldots, \a_{m}\})\cap \overline{B(0,r_m)}=\tilde X_m\subseteq \QQ$;
\item[(vii)] If $P_{m-1}(\a_m)\neq 0$ then $P_m(f_m(\a_m))\ne 0$; if $\tau\in \tilde X_m$ and $P_{m-1}(\tau)\ne 0$ then $\tau\not\in\{f_m(w): w\text{ is a root of } P_m(z)\}$.
\end{itemize}

The polynomials $h_n$ have the form $\sum_{j=0}^{s_n+\ell_n}\e_{n,j}z^{1-\hat{\d}_{j,0}}\tilde P_{n,j}(z)$, where
$s_n$ and $\ell_n$ are natural numbers to be chosen later. Here, $\hat{\d}_{j,0}$ is $1$ if $j=0$ and $0$ otherwise, and $\tilde P_{n,j}$ are monic polynomials with complex coefficients.

The required function will have the form $f(z)=\lim_{m\to \infty}f_m(z)$, where, for each $m\ge 1$, $f_{m+1}(z)=g(z)+\sum_{1\le n\le m}\sum_{j=0}^{s_n+\ell_n}\e_{n,j}z^{n+2-\hat{\d}_{j,0}}P_{n,j}(z)$, with $P_{n,j}(z)=P_n(z)\tilde P_{n,j}(z)$. Thus
\[
f(z)=g(z)+\e_0+\sum_{n\geq 2}\sum_{j=0}^{s_n+\ell_n}\e_{n,j}z^{n+2-\hat{\d}_{j,0}}P_{n,j}(z).
\]
We will have, for $0\le j<s_n+\ell_n$, that $P_n(z)\mid P_{n,j}(z)\mid P_{n,j+1}(z)\mid P_{n+1}(z)$.

At each step, we shall choose $\e_{n,j}$ with
\begin{equation}\label{ed}
    0<|\e_{n,j}|<\nu_{n,j},
\end{equation}
where
\begin{center}
$\nu_{n,j}:=\frac{\Theta(n,j)}{L(P_{n,j})2^{j+2} n^{n+2+\deg P_{n,j}}}$ \;and\; $\Theta(n,j):=\min_{k\in [n+1,n+2+\deg P_{n,j}]}\{\theta_k\}<1$.
\end{center}
Since $|P_{n,j}(z)|\leq L(P_{n,j})\max\{1,|z|\}^{\deg P_{n,j}}$ and $|z^{n+2-\hat{\d}_{j,0}}|\leq \max\{1,|z|\}^{n+2}$, we have the following for $z\in B(0,R)$ and $n\ge \max\{1,R\}$:
\[
\left|\sum_{j=0}^{s_n+\ell_n}\e_{n,j}z^{n+2-\hat{\d}_{j,0}}P_{n,j}(z)\right|<\sum_{j=0}^{s_n+\ell_n}\frac1{2^{j+2}}\left(\frac{\max\{1,R\}}{n}\right)^{n+2+\deg P_{n,j}}<\left(\frac{\max\{1,R\}}{n}\right)^{n+2}.
\]
Thus the series \[\sum_{n\geq 2}\sum_{j=0}^{s_n+\ell_n}\e_{n,j}z^{n+2-\hat{\d}_{j,0}}P_{n,j}(z)\] converges uniformly in any of these balls (note that $\deg P_{n,j}$ is a non-decreasing function in $j$).

\subsubsection{Inductive construction of $f_{n+1}$ from $f_n$} Supposing that we have a function $f_{n}$ satisfying (i)-(vii), we now construct $f_{n+1}$ with these properties. 

Before continuing with the proof, we shall make some conventions in order to make the text simpler and clearer. For a positive real number $\e$, we say that $h_1(z)$ is an {\it $\e$-perturbation} of $h_2(z)$ if $(h_2(z)-h_1(z))/\e$ is a nonzero polynomial with coefficients bounded by $1$ in norm. Note that all the intermediate functions (from $f_n(z)$ to $f_{n+1}(z)$) will have the form $f_J(z)=f_{J-1}(z)+\e_{J} z^{n+2}P_J(z)$, where $P_J(z)$ is a polynomial and $f_{J-1}\in \mathcal{E}$. We then shall say that a complex number $\gamma$ is {\it nailed} in $f_J(z)$ if $P_J(\gamma)=0$. Moreover, we define two types of periodic orbits: a $k$-periodic orbit $\{\b, f_J(\b),\ldots, f^{k-1}_J(\b)\}$ is a {\it nailed} orbit if $f_J^{j}(\b)$ is nailed, for all $j\in [0,k-1]$. On the other hand, the orbit is said to be {\it free} if $f_J^{j}(\b)$ is not nailed, for all $j\in [0,k-1]$. Also, by the definition of $P_J(z)$, a nailed orbit must be algebraic.

Let $B_n=B(0,r_n)$. We define $f_{n,0}(z)$ as
\[
f_{n,0}(z)=f_n(z)+\e_{n,0}z^{n+1}P_{n,0}(z),
\]
for a suitably chosen small positive real number $\e_{n,0}$ (here $P_{n,0}(z)=P_n(z)$). In particular, $f_{n,0}$ will enjoy the following properties:
$$(1)\,\,f_{n,0}^{-1}(\alpha_i)\cap B_n=f_n^{-1}(\alpha_i)\cap B_n, \forall i\in [1, n];$$
$$(2)\,\,f'_{n,0}(w)=f'_n(w)\neq 0, \forall w\in f_{n,0}^{-1}(\{\a_1,\a_2,\ldots, \a_{n+1}\}).$$

By property (vi) of the inductive hypothesis, one has that \break $f_n^{-1}(\{\a_1,\a_2,\ldots, \a_n\})\cap \partial B_n=\emptyset$, and any $\tau$ belonging to \break $f_n^{-1}(\{\a_1,\a_2,\ldots, \a_n\})\cap B_n$ is an algebraic number for which $f'_n(\tau)\neq 0$ (it follows that $\tau$ is a double zero of $P_{n}$, by property (ii)). Since $\e_{n,0}$ is small, then the number of zeros (counted with multiplicity) of $f_n(z)-\a_i$ and $f_{n,0}(z)-\a_i$ belonging to $B_n$ are equal by Lemma \ref{l6}. However, every zero of $f_n(z)-\a_i$ in $B_n$ is a zero of $f_{n,0}(z)-\a_i$ and every zero of $f_n(z)-\a_i$ in $B_n$ is simple. Therefore $f_{n,0}^{-1}(\a_i)\cap B_n=f_{n}^{-1}(\a_i)\cap B_n$ for all $i\in [1,n]$. (This argument ensures that, in our construction, no new preimage under $f_{n,0}$ of $\{\a_1,\ldots, \a_n\}$ lying in $B_{n+1}$ will appear apart from preimages under $f_n$.)

Next we show that, except for a countable set of values of $\e_{n,0}$, for any $w\in \C$ with $f_{n,0}(w)\in \{\a_1,\a_2,\ldots, \a_{n+1}\}$ we have $f_{n,0}'(w)\ne 0$, and for each $j\in [1,n+1]$, there is $w\in\C$ with $f_{n,0}(w)=\a_j$. Notice that for each $j\in [1,n+1]$, there is at most one value of $\e_{n,0}$ for which the image of $f_{n_0}$ does not contain $\a_j$ (since $\frac{\a_j-f_n(z)}{z^{n+1}P_{n,0}(z)}$ has an essential singularity at $\infty$). If $w$ is a root of $P_{n,0}(z)$, then $f'_{n,0}(w)=f'_n(w)\neq 0$ (notice that, if $w\in Y_n$, then $f'_n(w)\neq 0$, since the orbit of $w$ by iterations of $f_n$ is periodic and repelling). Then $w$ is a simple root of $f_n(z)-\a_i$. Hence, we may assume that $P_{n,0}(w)\ne 0$. We have $f_{n,0}(z)=f_n(z)+\e_{n,0}g(z)$, where $g(z)=z^{n+1}P_{n,0}(z)$. If $f_{n,0}'(w)=0$, we should have $f_n(w)+\e_{n,0}g(w)=\a_i$ and $f_n'(w)+\e_{n,0}g'(w)=0$. Defining $h_i(z)=f_n(z)-\a_i$, we have $h_i'(z)=f_n'(z)$, and thus \begin{center}
$h_i(w)+\e_{n,0}g(w)=0$ \;and\; $h_i'(w)+\e_{n,0}g'(w)=0$.
\end{center}
Now let $\psi_i(z)=-h_i(z)/g(z)$. Since $h_i(w)+\e_{n,0}g(w)=0$, we have $\psi_i(w)=-h_i(w)/g(w)=\e_{n,0}$. Moreover, $\psi_i'(w)=\frac{h_i(w)g'(w)-h_i'(w)g(w)}{g(w)^2}=0$ since $h_i(w)g'(w)-h_i'(w)g(w)=(h_i(w)+\e_{n,0}g(w))g'(w)-(h_i'(w)+\e_{n,0}g'(w))g(w)=0$. This implies that $\e_{n,0}$ is a singular value of $\psi_i(w)$, and the set of singular values of a meromorphic function is countable; this concludes the argument.

Since $f_{n,0}(z)$ is a transcendental function, we can apply Lemma \ref{l5} to ensure, for a suitable small positive real value of $\e_{n,0}$ the existence of a real number $r_{n+1}>\max\{n+1,r_n\}$ such that $f^{-1}_{n,0}(\{\a_1,\ldots, \a_{n+1}\})$ does not intersect $\partial B(0,r_{n+1})$, where $B_{n+1}:=B(0,r_{n+1})$ contains at least $n+1+D_n$ repelling $k$-periodic orbits of $f_{n,0}(z)$, for all $k\in [2,n+1]$ and at least $n+1+D_n$ fixed points whose eigenvalues do not belong to $\{0,1\}$, where $D_n$ is the number of distinct roots of $P_n(z)$ (notice that the above properties $(1)$ and $(2)$ hold for every $\e_{n,0}$ in a suitable small interval $(0,\delta)$, where we may assume $\delta<\nu_{n,0}$, minus a countable set; Lemma \ref{l5} gives a subinterval $(\delta_1,\delta_2)$ of $(0,\delta)$ of values of $\e_{n,0}$ satisfying the above properties). It follows that $B_{n+1}$ contains at least $n+1$ free repelling $k$-periodic orbits of $f_{n,0}(z)$, for all $k\in [2,n+1]$ and at least $n+1$ free fixed points whose eigenvalues do not belong to $\{0,1\}$.

Now we will prove property (v). Let $c_m$ be the coefficient of $z^m$ in $\sum_{n=2}^m\sum_{j=0}^{s_n+\ell_n}\e_{n,j}z^{n+2-\hat{\d}_{j,0}}P_{n,j}(z)$. It is enough to show that, for any $m\ge 1$, $|c_m|<\theta_m$ (notice that if $n>m$ then all monomials in $z^{n+2-\hat{\d}_{j,0}}P_{n,j}(z)$ have degree larger than $m$). In order to prove this, observe that the modulus of the coefficient of $z^m$ in $z^{n+2-\hat{\d}_{j,0}}P_{n,j}(z)$ is bounded by $L(P_{n,j})$, so, since $$|\e_{n,j}|<\nu_{n,j}=\frac{\Theta(n,j)}{L(P_{n,j})2^{j+2} n^{n+2+\deg P_{n,j}}},$$
and the coefficient of $z^m$ in $z^{n+2-\hat{\d}_{j,0}}P_{n,j}(z)$ can only be nonzero if $n+1\le m\le n+2+\deg P_{n,j}$,
we have that $|c_m|$ is at most
$$\sum_{n=2}^m\sum_{\genfrac{}{}{0pt}{}{j\in [0,s_n+\ell_n]}{n+1\le m\le n+2+\deg P_{n,j}}}\frac{\Theta(n,j)}{2^{j+2} n^{n+2+\deg P_{n,j}}}.$$
Since, if $n+1\le m\le n+2+\deg P_{n,j}$, $\Theta(n,j)=\min_{k\in [n+1,n+2+\deg P_{n,j}]}\{\theta_k\}\le\theta_m$, we conclude that 
$$|c_m|\le \theta_m\cdot \sum_{n=2}^m\sum_{\genfrac{}{}{0pt}{}{j\in [0,s_n+\ell_n]}{n+1\le m\le n+2+\deg P_{n,j}}}\frac{1}{2^{j+2} n^{n+2+\deg P_{n,j}}}<$$
$$\theta_m\cdot \sum_{n=2}^m\sum_{j\ge 0}\frac{1}{2^{j+2} n^{n+2}}<\theta_m\cdot \sum_{n\ge 2}\frac{1}{ n^{n+2}}<\theta_m,$$
which concludes the proof.

Next, we define $f_{n,1}(z)$ by  
\[
f_{n,1}(z)=f_{n,0}(z)+\e_{n,1}z^{n+2}P_{n,1}(z)
\]
(where $P_{n,1}(z)$ is also equal to $P_{n,0}(z)$). Note that $f'_{n,1}(y)\neq 0$, for all $y\in \{\a_1,\ldots,\a_{n}\}\cup (f_n^{-1}(\{\a_1,\a_2,\ldots, \a_{n+1}\}\cap B_{n+1})$, by choosing $\e_{n,1}$ to have sufficiently small norm. Indeed, since $f_{n,1}(z)-f_n(z)$ is a multiple of $z^2P_n(z)$, which is a multiple of $(x-\alpha_j)^2$ for $1\le j\le n$, it follows that $f_{n,1}'(\alpha_j)=f_n'(\alpha_j)\neq 0$ for $1\le j\le n$. On the other hand, since $f_{n,0}'(\tau)\neq 0, \forall \tau\in f_n^{-1}(\{\a_1,\a_2,\ldots, \a_{n+1}\})$ (and in particular for every $\tau$ in the finite set $f_n^{-1}(\{\a_1,\a_2,\ldots, \a_{n+1}\}\cap B_{n+1}$), it follows that, for $|\e_{n,1}|$ small enough, $f_{n,1}(z)=f_{n,0}(z)+\e_{n,1}z^{n+2}P_{n,1}(z)$ also satisfies $f_{n,1}'(\tau)\neq 0, \forall \tau\in f_n^{-1}(\{\a_1,\a_2,\ldots, \a_{n+1}\}\cap B_{n+1}$.

If $\a_{n+1}$ is a root of $P_{n,1}(z)$, then $f_{n,1}(\a_{n+1})=f_{n,0}(\a_{n+1})\in \QQ$, by the definition of the roots of $P_n(z)$ in (ii). Hence, in this case, we are done. Thus, suppose that $\a_{n+1}$ is not a zero of $P_{n,1}(z)$. Then, by density of $\QQ$, there exists such an $\e_{n,1}$ such that
\[
f_{n,1}(\a_{n+1})=f_{n,0}(\a_{n+1})+\e_{n,1}\a_{n+1}^{n+2}P_{n,1}(\a_{n+1})\in \QQ.
\]

Let $f_{n,0}^{-1}(\{\a_1,\a_2,\ldots, \a_{n+1}\})\cap (B_{n+1}\setminus\{0\})=\{\t_1,\t_2,\dots,\t_{m_n}\}$. Note that each $\a_i$ with $1\le i\le n+1$ has at least one preimage in $\{\t_1,\t_2,\dots,\t_{m_n}\}$. For $j\in [1, m_n]$, fix a small ball $B(\t_j,\eta_j)$ which does not intersect $\partial B_{n+1}$ in such a way that the balls $B(\t_j,\eta_j)$ are disjoint. The number $s_n$ of further steps in the construction of $f_{n+1}$ will be the number of elements of $(\{\a_{n+1}\}\cup \{\t_1,\t_2,\dots,\t_{m_n}\})\setminus(X_n\cup \tilde X_n)$. The following $s_n$ perturbations $f_{n,j}$ (with $j\in [1, s_n]$) of $f_{n,0}$ will be taken so close to $f_{n,0}$ that they will be $\e$-perturbations (as defined before, with $\e$ sufficiently small) with the further requirement that the number of zeros (counted with multiplicity) of $f_{n,0}(z)-\a_i$ and $f_{n,j}(z)-\a_i$ in the ball $B(\t_j,\eta_j)$ ($j\le m_n$) are equal, for all $i\in [1,n+1]$ and $j\in [1,s_n]$ (from now on, $\e$ will be called \emph{admissible} by including this extra requirement). Notice that in the balls $B(\t_j,\eta_j), j\le m_n$ these numbers are equal to one, since $\t_j$ is a simple zero of $f_{n,0}(z)-\a_i$ for some $i\in [1,n+1]$. We will take $f_{n+1,0}:=f_{n,s_n}$, and so the number of zeros (counted with multiplicity) of $f_{n,0}(z)-\a_i$ and $f_{n+1,0}(z)-\a_i$ in the balls $B_{n+1}$, $B_n$ and $B(\t_j,\eta_j), j\le m_n$ will be equal, for all $i\in [1,n+1]$. In particular, as before, $f_{n+1,0}^{-1}(\a_i)\cap B_n= f_{n}^{-1}(\a_i)\cap B_n$, for all $i\in [1,n]$ and, for $j\in [1, m_n]$, $f_{n+1,0}$ has only one zero in $B(\t_j,\eta_j)$, which is simple.

Define 
\begin{center}$f_{n,2}(z)=f_{n,1}(z)+\e_{n,2}z^{n+2}P_{n,2}(z)$,
\end{center}
where $P_{n,2}(z)=P_{n,1}(z)(z-\a_{n+1})$ and set $g_{n,2}(z)=z^{n+2}P_{n,2}(z)$. Since we are supposing that $\a_{n+1}$ is not a root of $P_{n,1}(z)$, then $g'_{n,2}(\a_{n+1})\neq 0$ with at most one exception; we can then choose $\e_{n,2}$ admissible (and with small norm) such that $f'_{n,2}(\a_{n+1})\neq 0$.

Let $J$ be the set of indices $j\in [1, m_n]$ such that $\t_j$ does not belong to $\{\a_{n+1}\}\cup X_n\cup \tilde X_n\cup \{0\}$. Let $J=\{j_1,j_2,\dots,j_{s_n-1}\}$ (notice that $J$ has $s_n-1$ elements since $s_n$ is the number of elements of $(\{\a_{n+1}\}\cup \{\t_1,\t_2,\dots,\t_{m_n}\})\setminus(X_n\cup \tilde X_n)$; removing the element $\a_{n+1}$ from this set, we get the $s_n-1$ elements of $J$). For each $i\in [1, s_n-1]$, we will do a perturbation as below. Let $z_i$ be the only element of $f_{n,1+i}^{-1}(\{\a_1,\ldots, \a_{n+1}\})$ in $B(\t_{j_i},\eta_{j_i})$. In the $i$-th step we will guarantee that the element of $f_{n,2+i}^{-1}(\{\a_1,\ldots, \a_{n+1}\})$ in $B(\t_{j_i},\eta_{j_i})$ will belong to $\QQ$, and in subsequent steps the images of these elements by the maps $f_{n,1+j}$ will remain unchanged.

Define $f_{n,3}(z)$ by
\[
f_{n,3}(z)=f_{n,2}(z)+\e_{n,3}z^{n+2}P_{n,3}(z),
\]
where $P_{n,3}(z)=P_{n,2}(z)$. To simplify, we write $\e:=\e_{n,3}$ and $F(\e,z)=f_{n,3}(z)$. Let $w=f_{n,2}(z_1)\in \{\a_1,\ldots,\a_{n+1}\}$. For each algebraic number $\t$ close to $z_1$, let 
$$\e(\t)=\frac{w-f_{n,2}(\t)}{\t^{n+2}P_{n,3}(\t)}$$
be such that $f_{n,3}(\tau)=w$. Since $\frac{\partial F}{\partial z}(0,z_1)=f_{n,2}'(z_1)\neq 0$, for $\t$ close enough to $z_1$, we have $f_{n,3}'(\tau)\ne 0$. We may also choose $\tau$ not belonging to the set $\{f_{n,2}(z): z\text{ is a root of } P_{n,3}\}$. 

Thus, following this construction we will obtain at the end a function $f_{n,s_n}$ such that $f_{n,s_n}^{-1}(\{\a_1,\ldots, \a_{n+1}\})\cap (B_{n+1} \cup (\cup_{j=1}^{m_n}B(\t_j,\eta_j)))$ is a subset of $\QQ$. We set $f_{n+1,0}(z):=f_{n,s_n}(z)$. This function will satisfy the items (i), (iii)-(vi) (with $n$ replaced by $n+1$). 

To finish the construction, we need to deal with the periodic orbits of $f_{n+1,0}$ lying in $B_{n+1}$. Since at each step we have taken $\e$-perturbations, there must exist at least $n+1$ $k$-periodic orbits of $f_{n+1,0}$ lying inside of $B_{n+1}$, for all $k\in [1,n+1]$. Moreover, if $\a\in Y_n$, then $P_{n,s_n-1}(f_n^j(\a))=0$, for all $j\in [0,k]$ and so $f_{n+1,0}^j(\a)=f_n^j(\a)$ for all $j\in [1,k]$ (in other words, the whole orbit of $\a$ is nailed). 

Suppose that
\[\{\b,f_{n+1,0}(\b),\ldots, f_{n+1,0}^{k-1}(\b)\}\]
is a free repelling $k$-periodic orbit of $f_{n+1,0}(z)$, for some $k\in [1,n+1]$. Notice that $\b$ is not nailed (as well as all the elements of its orbit). Take algebraic numbers $\g_0$ very close to $\b$ and $\g_1$ very close to $f_{n+1,0}(\b)$ and let $P_{n+1,1}(z):=P_{n+1,0}(z)$. Since $P_{n+1,1}(\g_0)\neq 0$, there exists a very small and admissible $\e_{n+1,1}$ for which $f_{n+1,0}(\g_0)+\e_{n+1,1}(\g_0)^{n+2}P_{n+1,1}(\g_0)=\g_1$. 
Now, we define $f_{n+1,2}(z)$ as
\[
f_{n+1,2}(z)=f_{n+1,1}(z)+\e_{n+1,2}z^{n+2}P_{n+1,2}(z),
\]
where $P_{n+1,2}(z):=P_{n+1,1}(z)(z-\g_0)^2$. We continue this construction until we arrive at a function $f_{n+1,k}(z)$ and algebraic numbers $\g_0,\ldots, \g_{k-1}$ such that $\g_i$ is very close to $f_{n+1,0}^{i}(\b)$ and $f_{n+1,k}(\g_j)=\g_{j+1}$, for all $(i,j)\in [0,k-1]\times [0,k-2]$.
Finally, we define $f_{n+1,k+1}(z)$ as
\[
f_{n+1,k+1}(z)=f_{n+1,k}(z)+\e_{n+1,k+1}z^{n+2}P_{n+1,k+1}(z),
\]
where $$P_{n+1,k+1}(z):=P_{n+1,1}(z)(z-\g_0)^2\cdots (z-\g_{k-2})^2.$$ To finish, we take a small and admissible $\e_{n+1,k+1}$ for which $f_{n+1,k+1}(\g_{k-1})=\g_0$. This construction works since our original periodic orbit is free, and so we are able to create new algebraic periodic orbits of period $k$ until reaching the required number (and we do not create undesirable nailed algebraic periodic orbits in the process).

\subsubsection{The final part}
By construction, the function 
\[
f(z)=\lim_{m\to \infty}f_m(z)=g(z)+\sum_{n\geq 1}\sum_{j=0}^{s_n+\ell_n}\e_{n,j}z^{n+2-\hat{\d}_{j,0}}P_{n,j}(z)=\sum_{n\geq 0} a_nz^n
\]
is entire and satisfies the desired properties: $f(\QQ)=\QQ, f^{-1}(\QQ)=\QQ, a_n\in K$ for all $n\ge 0$, and $\#\{\o(k,f)\}=s_k$ for all $k\geq 1$. Indeed, for each $m\ge 1$, $f_{k+1}(\a_m)=f_k(\a_m)\in \QQ$ for all $k$ such that $k\ge m$. In particular since $f=\lim_{k\to\infty}f_k$, we have $f(\a_m)=f_m(\a_m)\in \QQ$. Hence $f(\QQ)\subset \QQ$. On the other hand, since $f^{-1}_{m+1}(\{\a_1,\a_2,\ldots, \a_{m}\})\cap B(0,r_m)=f^{-1}_m(\{\a_1,\a_2,\ldots, \a_{m}\})\cap B(0,r_m)=f^{-1}_m(\{\a_1,\a_2,\ldots, \a_{m}\})\cap \overline{B(0,r_m)}=\tilde X_m\cap B(0,r_m)\subseteq \QQ$ and $B(0,r_m)\subset B(0,r_{m+1})$ for all $m\ge 1$, it follows by induction that, for each $k\ge m$,
\[f^{-1}_k(\{\a_1,\a_2,\ldots, \a_{m}\})\cap B(0,r_m)=f^{-1}_m(\{\a_1,\a_2,\ldots, \a_{m}\})\cap B(0,r_m).\]
This implies that 
\[
f^{-1}(\{\a_1,\a_2,\ldots, \a_{m}\})\cap B(0,r_m)=f^{-1}_m(\{\a_1,\a_2,\ldots, \a_{m}\})\cap B(0,r_m)\subseteq \QQ.
\]
Indeed, $f=\lim_{k\to\infty}f_k$, and so 
\[
f^{-1}(\{\a_1,\a_2,\ldots, \a_{m}\})\cap B(0,r_m)\supset f^{-1}_m(\{\a_1,\a_2,\ldots, \a_{m}\})\cap B(0,r_m).
\]
Now, if there were another element $w\in f^{-1}(\{\a_1,\a_2,\ldots, \a_{m}\})\cap B(0,r_m)$, it should be at a positive distance from the finite set $f^{-1}_m(\{\a_1,\a_2,\ldots, \a_{m}\})\cap B(0,r_m)$, but, since $f=\lim_{n\to \infty} f_n$, arbitrarily close to $w$ there should be, for large $k$, an element of $f^{-1}_k(\{\a_1,\a_2,\ldots, \a_{m}\})$ (also by Rouch\' e's theorem), which contradicts the equality 
\[
f^{-1}_k(\{\a_1,\a_2,\ldots, \a_{m}\})\cap B(0,r_m)=f^{-1}_m(\{\a_1,\a_2,\ldots, \a_{m}\})\cap B(0,r_m).
\]
Hence $f^{-1}(\QQ)\subset \QQ$ (and $f^{-1}(\QQ)=\QQ$, since $f(\QQ)\subset \QQ$).
Moreover, since, for each integer $j\in [1,n]$, there is $\tau\in \tilde X_n$ such that $f_n(\tau)=\a_j$, and, since $P_n(z)\mid P_{n+1}(z)\ \forall n\ge 1$ we shall have, for each $k\ge m$, that $f_k(\tau)=\a_j$ and so $f(\tau)=\a_j$. In conclusion, $f(\QQ)=\QQ$ and $f^{-1}(\QQ)=\QQ$. 

We also prove, in an analogous way, that $\#\{\o(k,f)\}=s_k$, $\forall k\geq 1$ (since, for every natural number $m$, $Y_m^{(k)}$ is a union of exactly $\min\{m,s_k\}$ algebraic periodic orbits of period $k$ of $f_m$ contained in $B(0,r_{m})$, and is contained in the set of roots of $P_m$).

Moreover, the function $f(z)=\sum_{n\geq 0}a_nz^n$ is transcendental, since it is a non-polynomial entire function (recall that $a_n\neq 0$ $\forall n\geq 0$). This completes the proof.\qed


\end{document}